\documentclass{amsart}
\usepackage[ocgcolorlinks,linktoc=all]{hyperref}
\hypersetup{citecolor=blue,linkcolor=red}
\newtheorem{theorem}{Theorem}[section]

\newtheorem{lemma}[theorem]{Lemma}
\newtheorem{proposition}[theorem]{Proposition}
\newtheorem{corollary}[theorem]{Corollary}
\theoremstyle{definition}

\theoremstyle{remark}

\numberwithin{equation}{section}

\begin{document}
\title[The $L_{-2}$ Minkowski problem]{A flow approach to the $L_{-2}$ Minkowski problem}
\author[M. N. Ivaki]{Mohammad N. Ivaki}
\address{Department of Mathematics and Statistics,
  Concordia University, Montreal, QC, Canada, H3G 1M8}
\curraddr{}
\email{mivaki@mathstat.concordia.ca}
\title[The $L_{-2}$ Minkowski problem]{A flow approach to the $L_{-2}$ Minkowski problem}
\date{}

\dedicatory{}

\begin{abstract}
We prove that the set of smooth, $\pi$-periodic, positive functions on the unit circle for which the planar $L_{-2}$ Minkowski problem is solvable is dense in the set of all smooth, $\pi$-periodic, positive functions on the unit circle with respect to the $L^{\infty}$ norm. Furthermore, we obtain a necessary condition on the solvability of the even $L_{-2}$ Minkowski problem. At the end, we prove uniqueness of the solutions up to special linear transformations.
\end{abstract}

\maketitle
\section{Introduction}
In differential geometry, the classical Minkowski problem concerns the existence, uniqueness and regularity of closed convex hypersurfaces whose Gauss curvature is prescribed as function of the normals. More generally, the Minkowski problem asks what are the necessary and sufficient conditions on a Borel measure on $\mathbb{S}^{n-1}$ to guarantee that it is the surface area measure of a convex body in $\mathbb{R}^n$.
If the measure $\mu$ has a smooth density $\Phi$ with respect to the Lebesgue measure of the unit sphere $\mathbb{S}^{n-1}$, the Minkowski problem is equivalent to the study of solutions to the following partial differential equation on the unit sphere
$$\det (\bar{\nabla}^2 s + {\hbox{Id}} \, s)=\Phi,$$
where $\bar{\nabla}$ is the covariant derivative  on $\mathbb{S}^{n-1} $ endowed with an orthonormal frame. Note that for a smooth convex body $K$ with support function $s$, the quantity $\det (\bar{\nabla}^2 s + {\hbox{Id}} \, s)$ is the reciprocal of the Gauss curvature of the boundary of $K.$
The answer to the existence and uniqueness of the Minkowski problem is as follows. If the support of $\mu$ is not contained in a great subsphere of $\mathbb{S}^{n-1}$, and it satisfies
$$\int_{\mathbb{S}^{n-1}}zd\mu(z)=0,$$
then it is the surface area of a convex body, and the solution is unique up to a translation. Minkowski himself solved the problem in the category of polyhedrons.
A. D. Alexandrov and others solved the problem in general, however, without any information about the regularity of the (unique) convex hypersurface. Around 1953, L. Nirenberg (in dimension three) and  A. V. Pogorelov (in all dimensions) solved the regularity problem in the smooth category independently. For references, one can see works by Minkowski \cite{M1,M2}, Alexandrov \cite{AD1,AD2,AD3}, Fenchel and Jessen \cite{FJ}, Lewy \cite{Lew1,Lew2}, Nirenberg \cite{Nir}, Calabi \cite{Calabi}, Cheng and Yau \cite{Yau}, Caffarelli et al. \cite{caff}, and others.

In Lutwak's development of Brunn-Minkowski-Firey theory \cite{Lutwak1,Lutwak}, it has been shown  that the Minkowski problem is part of a larger family of problems called the $L_p$ Minkowski problems. In the $L_p$ Brunn-Minkowski-Firey theory, Lutwak introduced the notion of the $L_p$ surface area. Therefore, it is natural to ask what are the necessary and sufficient conditions on a Borel measure on $\mathbb{S}^{n-1}$ which guarantee that it is the $L_p$ surface area measure of a convex body. For $p \geq 1$, and an even measure, existence and uniqueness of the convex body was established by Lutwak \cite{Lutwak1}. If the measure $\mu$ has a smooth density $\Phi$ with respect to the Lebesgue measure of the unit sphere $\mathbb{S}^{n-1}$, the $L_{p}$ problem is equivalent to the study of solutions to the following Monge-Amp\`ere equation on the unit sphere
$$ s^{1-p} \det (\bar{\nabla}^2 s + {\hbox{Id}} \, s)=\Phi,$$  where $\bar{\nabla}$ is the covariant derivative  on $\mathbb{S}^{n-1} $ endowed with an orthonormal frame.
 Note that for $p = 1$ this is the classical Minkowski problem. Solutions to many cases of these generalized problems followed later by Ai, Chou, Andrews, B\"or\"oczky, Chen, Wang, Gage, Guan, Lin, Jiang, Lutwak, Oliker, Yang, Zhang, Stancu, Umanskiy \cite{Chou1,BA4,BA3,BLYZ,Chen,Chou2,G1,G2,Gu,Jiang,LO,LYZ1,LYZ2,S0,S1,S2,U}. The progress in studying $L_p$ Minkowski problems has been extremely fruitful and resulted in many applications to functional inequalities \cite{CL,LYZ1,LYZ2,LYZ3,LYZ4,LYZ5}.
This unified theory relates many problems that, previously, were not connected. Note also that, for constant data $\Phi$, many $L_p$ problems were treated as self-similar solutions of geometric flows \cite{BA6,BA5,BA4,BA3,G1,G2} and others.

The cases $p=-n$ and $p=0$ are quite special and more difficult. The even case $p=0$ has been recently solved by B\"or\"oczky, Lutwak, Yang and Zhang \cite{BLYZ}. Many challenges remain for the problem with $p<1$ and, particularly, for negative $p$. The above partial differential equation with $p \in[-2, 0]$ and $n=2$ has been studied by Chen \cite{Chen} and more recently by Jiang \cite{Jiang} for $\Phi$ not necessarily positive. For $p\leq -2$, some existence results were obtained by Dou and Zhu including generalizing the result obtained by Jiang in the case $p=-2$, \cite{DZ}.

The smooth $L_ {-n}$ Minkowski problem is technically more complex, than the well-known counterpart, the Minkowski problem in the classical differential geometry. It is the problem which seeks necessary and sufficient conditions for the existence of a solution to a particular affine invariant, fully nonlinear partial differential equation. It is essential to say, the term {\em centro} in centro-affine differential geometry emphasizes that, contrary to affine differential geometry or classical differential geometry, Euclidean translations of an object in the ambient space are not allowed. This generates a bothersome obstacle for studying the $L_{-n}$ problem in full generality. Previous investigations of the $L_ {-n}$  Minkowski problem have been restricted to the even $L_ {-n}$ Minkowski problem, e.q., the problem in which it is assumed $\mu$
has the same values on antipodal Borel sets \cite{Chou1,Chen,Chou2,Jiang}.

Let $K$ be a compact, centrally symmetric, strictly convex body, smoothly embedded in $\mathbb{R}^2$. We denote the space of such convex bodies by $\mathcal{K}_{sym}$. Let
 $$x_K:\mathbb{S}^1\to\mathbb{R}^2,$$ be
the Gauss parametrization of $\gamma:=\partial K$, the boundary of $K\in \mathcal{K}_{sym}$, where the origin of the plane is chosen to coincide with the center of symmetry of the body. The support function of $\partial K$ is defined by
 $$s_{\partial K}(z):= \langle x_K(z), z \rangle,$$
for each $z=(\cos \theta, \sin \theta) \in\mathbb{S}^1$. We denote the curvature of $\partial K$ by $\kappa$ and, furthermore, the radius of curvature of the curve $\partial K$ by $\mathfrak{r}$, viewed now as functions on $[0, 2 \pi]$ identified with the unit circle. They are related to the support function by
 $$\frac{1}{\kappa(\nu^{-1}(z))}=\mathfrak{r}[s](\theta):=\frac{\partial^2}{\partial \theta^2}s(\theta)+s(\theta),$$
where $\theta$ is the angle parameter on $\mathbb{S}^1$ as above, and $\nu^{-1}$ is the inverse of the Gauss map of $K.$

Suppose $\Phi:\mathbb{S}^1\to\mathbb{R}^{+}$ is a smooth function. The planar $L_{-2}$ Minkowski problem is equivalent to the study of positive solutions to the following ordinary differential equation on $[0,2\pi]$:
$$ s(s_{\theta\theta}+s)^{\frac{1}{3}}=\Phi.$$
A positive solution to this equation corresponds to the existence of a convex body with support function $s$ and with {\em affine support function} $\Phi$. The function $\Phi$ is called even if $\Phi(z+\pi)=\Phi(z)$ for any $z\in\mathbb{S}^1.$

In this paper, we address the smooth even case of the $L_{-2}$ Minkowski problem. The main result obtained states that, although the $L_ {-2}$ Minkowski problem is not always solvable, we can always find functions that approximately solve the problem with any desired accuracy.
 We prove:
\begin{theorem}[Main theorem]\label{thm: main1}
Given an even, smooth function $\Phi:\mathbb{S}^1\to \mathbb{R}^{+}$, there exists a family of smooth convex bodies $\{K_n\}_{n\in\mathbb{N}}\subset \mathcal{K}_{sym},$ such that
$$\lim_{n\to \infty}\sup\limits_{\mathbb{S}^1}\left|\left|\frac{s}{\kappa^{1/3}}-\Phi\right|\right|=0.$$
Furthermore, if $\Phi$ is $\frac{\pi}{k}$ periodic for $k\geq2$, this family of convex bodies is uniformly bounded and
it converges in the $C^{\infty}$ norm to a smooth convex body whose support function satisfies $s(s_{\theta \theta} + s)^{1/3}=\Phi$.
 \end{theorem}
To prove our result, we will exploit an $\textbf{SL}(n)$-invariant curvature flow:\\
Let $K_0\in \mathcal{K}_{sym}$. We consider a family $\{K_t\}_t\subset \mathcal{K}_{sym}$ given by the smooth map $x:\mathbb{S}^1\times[0,T)\to \mathbb{R}^2$, which are evolving according to the $p$-weighted centro-affine flow namely,
 \begin{equation}\label{e: 1}
 \frac{\partial}{\partial t}x:=-\Psi(z) s\left(\frac{\kappa}{s^3}\right)^{\frac{p}{p+2}}\, z,~~
 x(\cdot,0)=x_{K_0}(\cdot)
 \end{equation}
for a fixed $p\in(1,2)$. Here $\Psi:\mathbb{S}^1\to\mathbb{R}^{+}$ is a smooth, even function. In this equation $x(\mathbb{S}^1,t)=\partial K_t$

Short time existence for the flow follows from the theory of parabolic partial differential equations. The flow itself, which is defined in a more generality by Stancu in \cite{S}, is new in the class of geometric evolution equations and displays many interesting properties. The long time behavior of the $p$-flow in $\mathbb{R}^2$ is settled by the author in \cite{Ivaki}. It was proved there that the volume preserving $p$-flow evolves any convex body in $ \mathcal{K}_{sym}$ to the unit disk in the Hausdorff distance, module $\textbf{SL}(2).$

\section{Convergence to a point}
In this section we prove that every solution of (\ref{e: 1}) starting from a smooth, symmetric convex body converges to a point in a finite time.

Let us denote the area of $K$ by $A:=A(K)=\displaystyle \frac{1}{2} \int_{\mathbb{S}^1}\frac{s}{\kappa}\, d\theta$ and the $p$-affine length of $K$ by $\Omega_p:=\Omega_p(K)=\displaystyle \int_{\mathbb{S}^1}\frac{s}{\kappa}\left(\frac{\kappa}{s^3}\right)^{\frac{p}{p+2}}d\theta.$
The following evolution equations can be derived by a direct computation.

\begin{lemma}\label{lem: evolution equations} Under the flow (\ref{e: 1}), one has
\begin{equation}\label{e: evolution equation}
\frac{\partial}{\partial t}\mathfrak{r}=-\frac{\partial^2}{\partial\theta^2}\left(\Psi
s^{1-\frac{3p}{p+2}}\mathfrak{r}^{-\frac{p}{p+2}}\right)- \Psi
s^{1-\frac{3p}{p+2}}\mathfrak{r}^{-\frac{p}{p+2}},
\end{equation}
and
\begin{equation}\label{e: volume}
\frac{d}{d
t}A=-\int_{\mathbb{S}^1}\Psi\frac{s}{\kappa}\left(\frac{\kappa}{s^3}\right)^{\frac{p}{p+2}}d\theta.
\end{equation}
\end{lemma}

 \begin{proposition} \label{cor: strengthned speed}
 The flow (\ref{e: 1}) increases in time the quantity
$$\displaystyle\min_{\theta\in\mathbb{S}^1}\left(\Psi s\left(\frac{\kappa}{s^3}\right)^{\frac{p}{p+2}}\right)(\theta,t).$$
\end{proposition}
\begin{proof}
Using the evolution equations  (\ref{e: 1}) and (\ref{e: evolution equation}), we obtain
\begin{align}\label{e: laplacian}
\frac{\partial}{\partial t}\left(\Psi s^{1-\frac{3p}{p+2}}\mathfrak{r}^{-\frac{p}{p+2}}\right)&=\Psi\left[\left(\frac{\partial}{\partial t}s^{1-\frac{3p}{p+2}}\right)\mathfrak{r}^{-\frac{p}{p+2}}+s^{1-\frac{3p}{p+2}}\frac{\partial}{\partial t}\mathfrak{r}^{-\frac{p}{p+2}}\right]\nonumber\\
&=-\left(1-\frac{3p}{p+2}\right)\Psi
s^{-\frac{3p}{p+2}}\mathfrak{r}^{-\frac{p}{p+2}}\left(\Psi
s^{1-\frac{3p}{p+2}}
\mathfrak{r}^{-\frac{p}{p+2}}\right)\\
&+\frac{p}{p+2}\Psi\mathfrak{r}^{-\frac{p}{p+2}-1}s^{1-\frac{3p}{p+2}}\left[
\left(\Psi s^{1-\frac{3p}{p+2}}\mathfrak{r}^{-\frac{p}{p+2}}\right)_{\theta\theta}+\Psi s^{1-\frac{3p}{p+2}}\mathfrak{r}^{-\frac{p}{p+2}}\right] \nonumber\\
&=\left(\frac{3p}{p+2}-1\right)\Psi^2s^{1-\frac{6p}{p+2}}\mathfrak{r}^{-\frac{2p}{p+2}}+\frac{p}{p+2}\Psi^2s^{2-\frac{6p}{p+2}}
\mathfrak{r}^{-\frac{2p}{p+2}-1}\nonumber\\
&+\frac{p}{p+2}\Psi\mathfrak{r}^{-\frac{p}{p+2}-1}s^{1-\frac{3p}{p+2}}\left(\Psi
s^{1-\frac{3p}{p+2}}\mathfrak{r}^{-\frac{p}{p+2}}\right)_{\theta
\theta}. \nonumber
\end{align}
Applying the maximum principle proves the claim.
\end{proof}
Consequently, we have that
\begin{corollary}\label{cor:Convexity is preserved}
The convexity of the evolving curves is preserved as long as the flow exists.
\end{corollary}
\begin{proof}
By Proposition ~\ref{cor: strengthned speed}, we have that, as long as the flow exists,
$$\min_{\theta\in\mathbb{S}^1}\Psi s\left(\frac{\kappa}{s^3}\right)^{\frac{p}{p+2}}(\theta,t)\ge
\min_{\theta\in\mathbb{S}^1}\Psi
s\left(\frac{\kappa}{s^3}\right)^{\frac{p}{p+2}}(\theta,0).$$
From this, we conclude  that $\kappa$ remain strictly positive.
\end{proof}
\begin{lemma}\label{lem: volume goes to zero} For any solution to the flow (\ref{e: 1}), the area of $K(t)$, $A(t)$, converges to zero in a finite time $T'$.
\end{lemma}
\begin{proof} By (\ref{e: volume}), we have
\begin{align*}
\frac{d}{d
t}A=-\int_{\mathbb{S}^1}\Psi\frac{s}{\kappa}\left(\frac{\kappa}{s^3}\right)^{\frac{p}{p+2}}d\theta\leq
-\delta\int_{\mathbb{S}^1}\frac{1}{\kappa}=-\delta L,
\end{align*}
where we used Proposition ~\ref{cor: strengthned speed}, and
$\delta:=\min\limits_{\mathbb{S}^1}\Psi s\left(\frac{\kappa}{s^3}\right)^{\frac{p}{p+2}}(\theta,0).$ On the other hand, by the isoperimetric inequality, we have $L\geq \sqrt{4\pi A}.$ Therefore, we obtain that
\begin{align*}
\frac{d}{dt}A\leq -\delta\sqrt{4\pi A}.
\end{align*}
This last inequality implies
\begin{align*}
\frac{d}{dt}\sqrt{A}\leq -\frac{\delta\sqrt{4\pi
}}{2}
\end{align*}
from which we conclude the statement of the lemma.
\end{proof}
\begin{lemma}\label{lem:decreasing p-affine surface area}
Any solution of the flow (\ref{e: 1}) satisfies $\lim_{t\to T'}\Omega_p(t)=0$.
\end{lemma}
\begin{proof} From the $p$-affine isoperimetric inequality in $\mathbb{R}^2$, \cite{Lutwak}, we have
$$0 \leq \Omega_p^{2+p}(t)\leq
2^{2+p}\pi^{2p}A^{2-p} (t),$$
for any $p\ge 1.$

Therefore, the result is a direct consequence of Lemma \ref{lem: volume goes to zero}. We recall that we consider the flow (\ref{e: 1}) for $1<p<2$.
\end{proof}
\begin{proposition}\label{prop: length}
 Let $L(t)$ be the length of $\partial K_t$ as $K_t$ evolves under (\ref{e: 1}). Then $\lim_{t\to T'}L(t)=0.$
\end{proposition}
\begin{proof}
We observe that
 \begin{equation}\label{ie: trick}
\min_{\theta\in\mathbb{S}^1}\left(s\left(\frac{\kappa}{s^3}\right)^{\frac{p}{2+p}}\right)(\theta,t)\int_{\mathbb{S}^1}\frac{1}{\kappa}\,
d\theta\leq \Omega_{p}(t)=\int_{\mathbb{S}^1}\frac{s}{\kappa}\left(\frac{\kappa}{s^3}\right)^{\frac{p}{2+p}}\, d\theta.
 \end{equation}
Thus, by taking the limit as $t \to T'$ on both sides of inequality (\ref{ie: trick}), and considering Proposition \ref{cor: strengthned speed}, we obtain
$$\lim_{t\to T'}L(t)=\lim_{t\to T'}\int_{\mathbb{S}^1}\frac{1}{\kappa}d\theta=0.$$
\end{proof}
Following an idea from \cite{Tso}, we consider the evolution of a test function to obtain an upper bound on the speed of the flow as long as the inradius of the evolving curve is uniformly bounded from below.
\begin{lemma}\label{lem: upper bound for speed}
If there exists an $ r>0$ such that $s\ge r$ on $[0,T)$, then $\kappa$ is uniformly bounded from above on $[0,T)$.
\end{lemma}
\begin{proof}
Define $Y(x,t):=\frac{\Psi s^{1-\frac{3p}{p+2}}\mathfrak{r}^{-\frac{p}{p+2}}}{s-\rho}$, where $\rho=\frac{1}{2}r$. For convenience, we set $\alpha:={1-\frac{3p}{p+2}}$ and $\beta:={-\frac{p}{p+2}}$. At the point where the maximum of $Y$ occurs, we have
$$Y_{\theta}=0, ~~ Y_{\theta\theta}\leq 0,$$
hence we obtain
 \begin{equation}\label{ie: akhar}
\left(\Psi s^{\alpha}\mathfrak{r}^{\beta}\right)_{\theta\theta}+\Psi s^{\alpha}\mathfrak{r}^{\beta}\leq-\Psi\left(\frac{\rho s^{\alpha}\mathfrak{r}^{\beta}-s^{\alpha}\mathfrak{r}^{\beta+1}}{s-\rho}\right).
 \end{equation}
Calculating
$$\frac{\partial}{\partial t}Y= \Psi
\left(\frac{s^{\alpha}}{s-\rho} \frac{\partial
\mathfrak{r}^{\beta}}{\partial
t}+\frac{\mathfrak{r}^{\beta}}{s-\rho} \frac{\partial
s^{\alpha}}{\partial t} -\frac{s^{\alpha} \mathfrak{r}^{\beta}}{(s-
\rho)^2} \frac{\partial s}{\partial t}\right),
$$
and using equation (\ref{e: evolution equation}), and inequality (\ref{ie: akhar}), we infer that, at the point where the maximum of $Y$ is reached, we have
$$0\leq\frac{\partial}{\partial t}Y\leq\frac{\Psi}{s-\rho}\left[{\beta}s^{\alpha}\mathfrak{r}^{\beta-1}\left(\frac{\rho s^{\alpha}\mathfrak{r}^{\beta}- s^{\alpha}\mathfrak{r}^{\beta+1}}{s-\rho}\right)-\alpha\mathfrak{r}^{2\beta}s^{2\alpha-1}+ \frac{ s^{2\alpha}\mathfrak{r}^{2\beta}}{s-\rho}\right].$$
This last inequality gives
$$\beta\rho\kappa-\beta-\alpha+\alpha\rho\frac{1}{s}+1\ge 0.$$ Neglecting the non-positive term $\displaystyle \alpha\rho\frac{1}{s}$, we obtain
$$\beta\rho\kappa-\beta-\alpha+1\ge 0.$$
Note that $\displaystyle\alpha+\beta-1=-\frac{4p}{p+2}$, therefore $\displaystyle 0\le\kappa\leq \frac{4}{\rho}$, consequently, implying the lemma.
\end{proof}

\begin{lemma}
Let $T$ be the maximal time of existence of the solution to the flow (\ref{e: 1}) with a fixed initial body $K_0 \in\mathcal{K}_{sym}$, then $T=T'$.
\end{lemma}
\begin{proof}
From Proposition \ref{prop: length}, we know that $T\leq T'.$ Therefore if $T<T'$ we conclude that $A(t)$ has a uniform lower bound which implies that the inradius of the evolving curve is uniformly bounded from below by a constant. Now, Corollary \ref{cor:Convexity is preserved} guarantees a uniform lower bound on the curvature of the evolving curve in the time interval $[0,T).$ On the other hand, Lemma \ref{lem: upper bound for speed} implies a uniform upper bound on the curvature of the evolving curve. Thus, the evolution equation (\ref{e: 1}) is uniformly parabolic on $[0,T)$, and bounds on higher derivatives of the support function follows by \cite{K} and Schauder theory. Hence, we can extend the solution
after time $T$,  contradicting its definition.
\end{proof}
Therefore, we have proved:
\begin{theorem}
Let $T$ be the maximal time of existence of the solution to the flow (\ref{e: 1}) with a fixed initial body $K_0 \in\mathcal{K}_{sym}$, then $K_t$ converges to
the origin.
\end{theorem}

\section{Affine differential setting}
In what follows, we find it more appropriate to work in the affine setting and we will now recall several definitions from affine differential geometry. Let $\gamma:\mathbb{S}^1\to\mathbb{R}^2$ be an embedded strictly convex curve with the curve parameter $\theta$. Define $\mathfrak{g}(\theta):=[\gamma_{\theta},\gamma_{\theta\theta}]^{1/3}$, where, for two vectors $u, v$ in $\mathbb{R}^2$, $[u, v]$ denotes the determinant of the matrix with rows $u$ and $v$. The
affine arc-length is then given by
\begin{equation}\label{def: affine arclength}
\mathfrak{s}(\theta):=\int_{0}^{\theta}\mathfrak{g}(\xi)d\xi.
\end{equation}
Furthermore, the affine tangent vector $\mathfrak{t}$, the affine normal vector $\mathfrak{n}$, and the affine curvature are defined, in this order, as follows:
\begin{equation*}
\mathfrak{t}:=\gamma_{\mathfrak{s}},~~~
\mathfrak{n}:=\gamma_{\mathfrak{s}\mathfrak{s}},
~~~\mu:=[\gamma_{\mathfrak{s}\mathfrak{s}},
\gamma_{\mathfrak{s}\mathfrak{s}\mathfrak{s}}].
\end{equation*}
In the affine coordinate ${\mathfrak{s}}$, the following relations hold:
\begin{align}\label{e: some prop of affine setting}
[\gamma_{\mathfrak{s}},\gamma_{\mathfrak{s}\mathfrak{s}}]&=1 \nonumber\\
[\gamma_{\mathfrak{s}},\gamma_{\mathfrak{s}\mathfrak{s}\mathfrak{s}}]&=0\\
[\gamma_{\mathfrak{s}\mathfrak{s}\mathfrak{s}\mathfrak{s}},\gamma_{\mathfrak{s}}]&=\mu\nonumber.
\end{align}
Moreover, it can be easily verified that $\displaystyle \frac{\kappa}{s^3}= \frac{[\gamma_{\theta},\gamma_{\theta\theta}]}{[\gamma,\gamma_{\theta}]^3}=\frac{[\gamma_{\mathfrak{s}},\gamma_{\mathfrak{s}\mathfrak{s}}]}{[\gamma,\gamma_{\mathfrak{s}}
]^3}.$
Since
$[\gamma_{\mathfrak{s}},\gamma_{\mathfrak{s}\mathfrak{s}}]=1$,
we conclude that $\displaystyle\frac{\kappa}{s^3}=\frac{1}{[\gamma,\gamma_{\mathfrak{s}} ]^3}.$ The affine support function is defined by $\sigma:=\frac{s}{\kappa^{1/3}}$, see
\cite{BA1,NS}.

Let $K_0\in \mathcal{K}_{sym}$. We consider the family $\{K_t\}_t\in \mathcal{K}_{sym}$, and their associated smooth embeddings $x:\mathbb{S}^1\times[0,T)\to \mathbb{R}^2$, which are evolving according  to
 \begin{equation}\label{e: 2}
 \frac{\partial}{\partial t}x:=\Psi\sigma^{1-\frac{3p}{p+2}}\mathfrak{n},~~
 x(\cdot,0)=x_{K_0}(\cdot),~~ x(\cdot,t)=x_{K_t}(\cdot)
 \end{equation}
 for a fixed $1< p< 2$.
Observe that up to a time-dependent diffeomorphism the flow defined in (\ref{e: 2}) is equivalent to the flow defined by (\ref{e: 1}).

In terms of affine invariant quantities, the area and the weighted $p$-affine length of $K$ are
 $$A(K)=\frac{1}{2}\int_{\gamma}\sigma d\mathfrak{s},~~\Omega_{p}^{\Psi}(K):=\int_{\gamma}\Psi\sigma^{1-\frac{3p}{p+2}}d\mathfrak{s},$$
where here and thereafter $\gamma$ is the boundary curve of $K$.
\begin{lemma}\label{lem: 1}
Let us define $e$ to be the Euclidean arc-length and $\gamma^t:=\partial K_t$ be the boundary of a convex body $K_t$ evolving under the flow (\ref{e: 2}). Then the following evolution equations hold:
\begin{enumerate}
\item $\displaystyle \frac{\partial}{\partial t}z=\kappa^{\frac{2}{3}}\left(\Psi\sigma^{1-\frac{3p}{p+2}}\right)_{\mathfrak{s}}x_{e},$
\item $\displaystyle \frac{\partial}{\partial t}\Psi=\Psi_{\mathfrak{s}}\left(\Psi\sigma^{1-\frac{3p}{p+2}}\right)_{\mathfrak{s}},$
\item   $\displaystyle\frac{d}{dt} A=-\Omega_{p}^{\Psi },$
\item $\displaystyle\frac{\partial}{\partial t}\mathfrak{g}=\left(-\frac{2}{3}\Psi\sigma^{1-\frac{3p}{p+2}}\mu+
       \frac{1}{3}\left(\Psi\sigma^{1-\frac{3p}{p+2}}\right)_{\mathfrak{s}\mathfrak{s}}\right)\mathfrak{g},$
\item $ \displaystyle \frac{\partial}{\partial t}\mathfrak{t}=\left(-\frac{1}{3}\Psi\sigma^{1-\frac{3p}{p+2}}\mu-\frac{1}{3}\left(\Psi\sigma^{1-\frac{3p}{p+2}}
      \right)_{\mathfrak{s}\mathfrak{s}}\right)\mathfrak{t}+
\left(\Psi\sigma^{1-\frac{3p}{p+2}}\right)_{\mathfrak{s}}\mathfrak{n},$
\item $\displaystyle\frac{\partial}{\partial t}\sigma=-\frac{4}{3}\sigma^{1-\frac{3p}{p+2}}\Psi
       +\frac{1}{3}\sigma^{1-\frac{3p}{p+2}}\Psi\sigma_{\mathfrak{s}\mathfrak{s}}
       -\frac{1}{3}\left(\sigma^{1-\frac{3p}{p+2}}\Psi\right)_{\mathfrak{s}\mathfrak{s}}\sigma
       +\left(\sigma^{1-\frac{3p}{p+2}}\Psi\right)_{\mathfrak{s}}\sigma_{\mathfrak{s}},$
\end{enumerate}
and we have
\begin{align}\label{e: 6}
  \frac{d}{d t}\Omega_{p}^{\Psi}&=
   \frac{2(p-2)}{p+2}\int_{\gamma}\Psi^2\sigma^{1-\frac{6p}{p+2}} d\mathfrak{s}
   +\frac{18p^2}{(p+2)^3}\int_{\gamma}\Psi^2\sigma^{-\frac{6p}{p+2}}\sigma_{\mathfrak{s}}^2d\mathfrak{s}\\\nonumber
   &-\frac{2}{p+2}\int_{\gamma}\Psi_{\mathfrak{s}}^2\sigma^{2-\frac{6p}{p+2}}d\mathfrak{s}
   -\frac{12p}{(p+2)^2}\int_{\gamma}\Psi\Psi_{\mathfrak{s}}\sigma^{1-\frac{6p}{p+2}}\sigma_{\mathfrak{s}}d\mathfrak{s}.\nonumber
 \end{align}
\end{lemma}
  \begin{proof}
To prove the lemma, we will use repeatedly equations (\ref{e: some prop of affine setting}) without further mention. Recall that
 $d\mathfrak{s}=\kappa^{\frac{1}{3}}de=\mathfrak{r}^{\frac{2}{3}}d\theta.$ \\
\textbf{Proof of} (1): Since $\frac{\partial}{\partial t}z$ is a
tangent vector,
\begin{align*}
\frac{\partial}{\partial t}z&=\langle \frac{\partial}{\partial t}z, x_{e}\rangle x_{e}\\
&=- \langle z, \frac{\partial^2}{\partial e\partial t}x\rangle x_{e}\\
&=-\langle z, \frac{\partial}{\partial e}\left(\Psi\sigma^{1-\frac{3p}{p+2}}\mathfrak{n}\right)\rangle x_{e}\\
&=-\langle z, \mathfrak{n}\rangle \frac{\partial}{\partial e}\left(\Psi\sigma^{1-\frac{3p}{p+2}}\right)x_{e}\\
&=\kappa^{\frac{2}{3}}\left(\Psi\sigma^{1-\frac{3p}{p+2}}\right)_{\mathfrak{s}}x_{e}.\
\end{align*}
\textbf{Proof of} (2): By the evolution equation (1), we have
\begin{align*}
\frac{\partial}{\partial t}\Psi(z)
&=\kappa^{\frac{2}{3}}\Psi_{\theta}\left(\Psi\sigma^{1-\frac{3p}{p+2}}\right)_{\mathfrak{s}}.
\end{align*}
\textbf{Proof of} (3): Note that (3) has been proved in Lemma \ref{e: evolution equation}.\\
\textbf{Proof of} (4):
\begin{align*}
\frac{\partial}{\partial t}\mathfrak{g}^3&=\frac{\partial}{\partial t}\left[\gamma_{\theta},\gamma_{\theta\theta}\right]
=\left[\frac{\partial}{\partial t}\gamma_{\theta},\gamma_{\theta\theta}\right]
+\left[\gamma_{\theta},\frac{\partial}{\partial t}\gamma_{\theta\theta}\right].\\
\end{align*}
We have that
\begin{align*}
\left[\frac{\partial}{\partial t}\gamma_{\theta},\gamma_{\theta\theta}\right]&=\left[\frac{\partial}{\partial \theta}\left(\Psi\sigma^{1-\frac{3p}{p+2}}\gamma_{\mathfrak{s}\mathfrak{s}}\right),\gamma_{\theta\theta}\right]\\
&=\left[\mathfrak{g}\frac{\partial}{\partial \mathfrak{s}}\left(\Psi\sigma^{1-\frac{3p}{p+2}}\gamma_{\mathfrak{s}\mathfrak{s}}\right),\gamma_{\theta\theta}\right]\\
&=\mathfrak{g}\left[\left(\Psi\sigma^{1-\frac{3p}{p+2}}\right)_{\mathfrak{s}}\gamma_{\mathfrak{s}\mathfrak{s}}+
\Psi\sigma^{1-\frac{3p}{p+2}}\gamma_{\mathfrak{s}\mathfrak{s}\mathfrak{s}},\gamma_{\theta\theta}\right].
\end{align*}
Since $\displaystyle\frac{\partial^2}{\partial\theta^2}=\mathfrak{g}\mathfrak{g}_{\mathfrak{s}}\frac{\partial}{\partial
\mathfrak{s}}+\mathfrak{g}^2\frac{\partial^2}{\partial\mathfrak{s}^2}$, we further have
 $\gamma_{\theta\theta}=\mathfrak{g}^2\gamma_{\mathfrak{s}\mathfrak{s}}+\mathfrak{g}\mathfrak{g}_{\mathfrak{s}}\gamma_{\mathfrak{s}}$
and, therefore,
\begin{align*}
\left[\frac{\partial}{\partial
t}\gamma_{\theta},\gamma_{\theta\theta}\right]
&=\mathfrak{g}\left[\left(\Psi\sigma^{1-\frac{3p}{p+2}}\right)_{\mathfrak{s}}\gamma_{\mathfrak{s}\mathfrak{s}}+
\Psi\sigma^{1-\frac{3p}{p+2}}\gamma_{\mathfrak{s}\mathfrak{s}\mathfrak{s}},\mathfrak{g}^2\gamma_{\mathfrak{s}\mathfrak{s}}
+\mathfrak{g}\mathfrak{g}_{\mathfrak{s}}\gamma_{\mathfrak{s}}\right]\\
&=-\mathfrak{g}^2\mathfrak{g}_{\mathfrak{s}}\left(\Psi\sigma^{1-\frac{3p}{p+2}}\right)_{\mathfrak{s}}-
\mathfrak{g}^3\Psi\sigma^{1-\frac{3p}{p+2}}\mu.
\end{align*}
On the other hand, we have
\begin{align*}
\left[\gamma_{\theta},\frac{\partial}{\partial
t}\gamma_{\theta\theta}\right]&=\left[\mathfrak{g}\gamma_{\mathfrak{s}},
\frac{\partial^2}{\partial\theta^2}\left(\Psi\sigma^{1-\frac{3p}{p+2}}\gamma_{\mathfrak{s}\mathfrak{s}}\right)\right]\\
&=\left[\mathfrak{g}\gamma_{\mathfrak{s}},
\mathfrak{g}\mathfrak{g}_{\mathfrak{s}}\frac{\partial}{\partial
\mathfrak{s}}\left(\Psi\sigma^{1-\frac{3p}{p+2}}\gamma_{\mathfrak{s}\mathfrak{s}}\right)+
\mathfrak{g}^2\frac{\partial^2}{\partial\mathfrak{s}^2}\left(\Psi\sigma^{1-\frac{3p}{p+2}}\gamma_{\mathfrak{s}\mathfrak{s}}\right)\right]\\
&=\mathfrak{g}^2\mathfrak{g}_{\mathfrak{s}}\left(\Psi\sigma^{1-\frac{3p}{p+2}}\right)_{\mathfrak{s}}+
\mathfrak{g}^3\left(\Psi\sigma^{1-\frac{3p}{p+2}}\right)_{\mathfrak{s}\mathfrak{s}}-\mathfrak{g}^3\Psi\sigma^{1-\frac{3p}{p+2}}\mu.
\end{align*}
Hence, we conclude that
$$\frac{\partial}{\partial t}\mathfrak{g}^3=\mathfrak{g}^3\left(\Psi\sigma^{1-\frac{3p}{p+2}}\right)_{\mathfrak{s}\mathfrak{s}}-
2\mathfrak{g}^3\Psi\sigma^{1-\frac{3p}{p+2}}\mu,$$
which verifies our fourth claim.\\
\textbf{Proof of} (5): To prove the fifth claim, we observe that
\begin{equation}\label{e: commute}
\frac{\partial}{\partial t}\frac{\partial}{\partial
\mathfrak{s}}=\frac{\partial}{\partial\mathfrak{s}}\frac{\partial}{\partial
t}- \frac{1}{\mathfrak{g}}\frac{\partial\mathfrak{g}}{\partial
t}\frac{\partial}{\partial \mathfrak{s}}.
\end{equation}
By (\ref{e: commute}), we get
\begin{align*}
\frac{\partial}{\partial t}\mathfrak{t}&=\frac{\partial}{\partial t}\frac{\partial}{\partial \mathfrak{s}}\gamma\\
&=\frac{\partial}{\partial \mathfrak{s}}\left(\Psi\sigma^{1-\frac{3p}{p+2}}\gamma_{\mathfrak{s}\mathfrak{s}}\right)+\left(\frac{2}{3}\Psi\sigma^{1-\frac{3p}{p+2}}\mu-\frac{1}{3} \left(\Psi\sigma^{1-\frac{3p}{p+2}}\right)_{\mathfrak{s}\mathfrak{s}}\right)\mathfrak{t}\\
&=\left(\Psi\sigma^{1-\frac{3p}{p+2}}\right)_{\mathfrak{s}}\mathfrak{n}+\Psi\sigma^{1-\frac{3p}{p+2}}\gamma_{\mathfrak{s}\mathfrak{s}\mathfrak{s}}
+\left(\frac{2}{3}\Psi\sigma^{1-\frac{3p}{p+2}}\mu-\frac{1}{3}
\left(\Psi\sigma^{1-\frac{3p}{p+2}}\right)_{\mathfrak{s}\mathfrak{s}}\right)\mathfrak{t}.
\end{align*}
We note that $\gamma_{\mathfrak{s}\mathfrak{s}\mathfrak{s}}=-\mu \gamma_{\mathfrak{s}}$ ending the proof of (5).\\
\textbf{Proof of} (6): We now proceed to prove the sixth claim with
\begin{align*}
\frac{\partial}{\partial t}\sigma=\frac{\partial}{\partial t}\left[\gamma,\gamma_{\mathfrak{s}}\right]=\left[\frac{\partial}{\partial t}\gamma, \gamma_{\mathfrak{s}}\right]+\left[\gamma,\frac{\partial}{\partial t}\gamma_{\mathfrak{s}} \right].
\end{align*}
By the evolution equation (\ref{e: 2}), the evolution equation for $\mathfrak{t}$, and the identities $\sigma=[\gamma,\gamma_{\mathfrak{s}}]$ and
$\sigma_{\mathfrak{s}}=[\gamma,\gamma_{\mathfrak{s}\mathfrak{s}}]$,
we get that
\begin{align*}
\frac{\partial}{\partial t}\sigma&=\left[\Psi\sigma^{1-\frac{3p}{p+2}}\gamma_{\mathfrak{s}\mathfrak{s}},
\gamma_{\mathfrak{s}}\right]+\left[\gamma,
\left(-\frac{1}{3}\Psi\sigma^{1-\frac{3p}{p+2}}\mu-
\frac{1}{3}\left(\Psi\sigma^{1-\frac{3p}{p+2}}\right)_{\mathfrak{s}\mathfrak{s}}\right)\gamma_\mathfrak{s}
+\left(\Psi\sigma^{1-\frac{3p}{p+2}}\right)_{\mathfrak{s}}\gamma_{\mathfrak{s}\mathfrak{s}}\right]\\
&=-\Psi\sigma^{1-\frac{3p}{p+2}}-\frac{1}{3}\Psi\sigma^{2-\frac{3p}{p+2}}\mu-
\frac{1}{3}\left(\Psi\sigma^{1-\frac{3p}{p+2}}\right)_{\mathfrak{s}\mathfrak{s}}\sigma+
\left(\Psi\sigma^{1-\frac{3p}{p+2}}\right)_{\mathfrak{s}}\sigma_{\mathfrak{s}}\\
&=-\frac{4}{3}\Psi\sigma^{1-\frac{3p}{p+2}}
       +\frac{1}{3}\Psi\sigma^{1-\frac{3p}{p+2}}\sigma_{\mathfrak{s}\mathfrak{s}}
       -\frac{1}{3}\left(\Psi\sigma^{1-\frac{3p}{p+2}}\right)_{\mathfrak{s}\mathfrak{s}}\sigma
       +\left(\Psi\sigma^{1-\frac{3p}{p+2}}\right)_{\mathfrak{s}}\sigma_{\mathfrak{s}}
\end{align*}
where we used $\sigma_{\mathfrak{s}\mathfrak{s}}+\sigma\mu=1$ on the second line.\\
\textbf{Proof of} (\ref{e: 6}): The proof follows directly from (2), (4), (6) and arranging similar terms.
\begin{align*}
\frac{d}{dt}\Omega_p^{\Psi}=\underbrace{\int_{\gamma}\left(\frac{\partial}{\partial t}\Psi\right)\sigma^{1-\frac{3p}{p+2}}d\mathfrak{s}}_{I}+
\underbrace{\int_{\gamma}\Psi\left(\frac{\partial}{\partial t}\sigma^{1-\frac{3p}{p+2}}\right)d\mathfrak{s}}_{II}+
\underbrace{\int_{\gamma}\Psi\sigma^{1-\frac{3p}{p+2}}\frac{\partial}{\partial t}d\mathfrak{s}}_{III}.
\end{align*}
We use (2) to compute $I$:
\begin{align*}
I&=\int_{\gamma}\left(\Psi_{\mathfrak{s}}\left(\Psi\sigma^{1-\frac{3p}{p+2}}\right)_{\mathfrak{s}}\right)\sigma^{1-\frac{3p}{p+2}}d\mathfrak{s}
\nonumber\\
&=\int_{\gamma}\Psi_{\mathfrak{s}}^2\sigma^{2-\frac{6p}{p+2}}d\mathfrak{s}
+\left(1-\frac{3p}{p+2}\right)\int_{\gamma}\Psi\Psi_{\mathfrak{s}}\sigma^{1-\frac{6p}{p+2}}\sigma_{\mathfrak{s}}d\mathfrak{s}.
\end{align*}
To simplify $II$ we deploy (6) and integration by parts:
\begin{align*}
II&=\left(1-\frac{3p}{p+2}\right)\int_{\gamma}\Psi\sigma^{-\frac{3p}{p+2}}\left[-\frac{4}{3}\sigma^{1-\frac{3p}{p+2}}\Psi
+\frac{1}{3}\sigma^{1-\frac{3p}{p+2}}\Psi\sigma_{\mathfrak{s}\mathfrak{s}}
-\frac{1}{3}\left(\sigma^{1-\frac{3p}{p+2}}\Psi\right)_{\mathfrak{s}\mathfrak{s}}\sigma
+\left(\sigma^{1-\frac{3p}{p+2}}\Psi\right)_{\mathfrak{s}}\sigma_{\mathfrak{s}}\right]d\mathfrak{s}\\
&=\left(\frac{4p}{p+2}-\frac{4}{3}\right)\int_{\gamma}\Psi^2\sigma^{1-\frac{6p}{p+2}}d\mathfrak{s}
+2\left(\frac{p}{p+2}-\frac{1}{3}\right)\int_{\gamma}\Psi\Psi_{\mathfrak{s}}\sigma^{1-\frac{6p}{p+2}}\sigma_{\mathfrak{s}}d\mathfrak{s}\\
&-\left(\frac{p}{p+2}-\frac{1}{3}\right)\left(\frac{6p}{p+2}-1\right)\int_{\gamma}\Psi^2\sigma^{-\frac{6p}{p+2}}
\sigma_{\mathfrak{s}}^2d\mathfrak{s}
-\left(\frac{p}{p+2}-\frac{1}{3}\right)\int_{\gamma}\left(\sigma^{1-\frac{3p}{p+2}}\Psi\right)_{\mathfrak{s}}^2d\mathfrak{s}\\
&+\left(1-\frac{3p}{p+2}\right)\int_{\gamma}\Psi\Psi_{\mathfrak{s}}\sigma^{1-\frac{6p}{p+2}}\sigma_{\mathfrak{s}}d\mathfrak{s}
+\left(1-\frac{3p}{p+2}\right)^2\int_{\gamma}\Psi^2\sigma^{-\frac{6p}{p+2}}\sigma_{\mathfrak{s}}^2d\mathfrak{s}\\
&=\left(\frac{4p}{p+2}-\frac{4}{3}\right)\int_{\gamma}\Psi^2\sigma^{1-\frac{6p}{p+2}}d\mathfrak{s}
+2\left(\frac{p}{p+2}-\frac{1}{3}\right)\int_{\gamma}\Psi\Psi_{\mathfrak{s}}\sigma^{1-\frac{6p}{p+2}}\sigma_{\mathfrak{s}}d\mathfrak{s}\\
&-\left(\frac{p}{p+2}-\frac{1}{3}\right)\left(\frac{6p}{p+2}-1\right)
\int_{\gamma}\Psi^2\sigma^{-\frac{6p}{p+2}}\sigma_{\mathfrak{s}}^2d\mathfrak{s}
+\left(\frac{1}{3}-\frac{p}{p+2}\right)\left(1-\frac{3p}{p+2}\right)^2
\int_{\gamma}\Psi^2\sigma^{-\frac{6p}{p+2}}\sigma_{\mathfrak{s}}^2d\mathfrak{s}\\
&+2\left(\frac{1}{3}-\frac{p}{p+2}\right)\left(1-\frac{3p}{p+2}\right)
\int_{\gamma}\Psi\Psi_{\mathfrak{s}}\sigma^{1-\frac{6p}{p+2}}\sigma_{\mathfrak{s}}d\mathfrak{s}
+\left(\frac{1}{3}-\frac{p}{p+2}\right)\int_{\gamma}\Psi_{\mathfrak{s}}^2\sigma^{2-\frac{6p}{p+2}}d\mathfrak{s}\\
&+\left(1-\frac{3p}{p+2}\right)\int_{\gamma}\Psi\Psi_{\mathfrak{s}}\sigma^{1-\frac{6p}{p+2}}\sigma_{\mathfrak{s}}d\mathfrak{s}
+\left(1-\frac{3p}{p+2}\right)^2\int_{\gamma}\Psi^2\sigma^{-\frac{6p}{p+2}}\sigma_{\mathfrak{s}}^2d\mathfrak{s}.
\end{align*}
To calculate $III$ we use (4), integration by parts and the identity $\sigma_{\mathfrak{s}\mathfrak{s}}+\sigma\mu=1$:
\begin{align*}
III&=\int_{\gamma}\Psi\sigma^{1-\frac{3p}{p+2}}\left(-\frac{2}{3}\Psi\sigma^{1-\frac{3p}{p+2}}\mu+
       \frac{1}{3}\left(\Psi\sigma^{1-\frac{3p}{p+2}}\right)_{\mathfrak{s}\mathfrak{s}}\right)d\mathfrak{s}\\
&=-\frac{2}{3}\int_{\gamma}\Psi^2\sigma^{2-\frac{6p}{p+2}}\mu d\mathfrak{s}
-\frac{1}{3}\int_{\gamma}\left(\Psi\sigma^{1-\frac{3p}{p+2}}\right)_{\mathfrak{s}}^2d\mathfrak{s}\\
&=-\frac{2}{3}\int_{\gamma}\Psi^2\sigma^{1-\frac{6p}{p+2}}(1-\sigma_{\mathfrak{s}\mathfrak{s}}) d\mathfrak{s}
-\frac{1}{3}\int_{\gamma}\left(\Psi\sigma^{1-\frac{3p}{p+2}}\right)_{\mathfrak{s}}^2d\mathfrak{s}\\
&=-\frac{2}{3}\int_{\gamma}\Psi^2\sigma^{1-\frac{6p}{p+2}} d\mathfrak{s}
-\frac{4}{3}\int_{\gamma}\Psi\Psi_{\mathfrak{s}}\sigma^{1-\frac{6p}{p+2}}\sigma_{\mathfrak{s}}d\mathfrak{s}
+\left(\frac{4p}{p+2}-\frac{2}{3}\right)\int_{\gamma}\Psi^2\sigma^{-\frac{6p}{p+2}}\sigma_{\mathfrak{s}}^2d\mathfrak{s}
\\
&-\frac{1}{3}\int_{\gamma}\left(\Psi\sigma^{1-\frac{3p}{p+2}}\right)_{\mathfrak{s}}^2d\mathfrak{s}\\
&=-\frac{2}{3}\int_{\gamma}\Psi^2\sigma^{1-\frac{6p}{p+2}} d\mathfrak{s}
-\frac{4}{3}\int_{\gamma}\Psi\Psi_{\mathfrak{s}}\sigma^{1-\frac{6p}{p+2}}\sigma_{\mathfrak{s}}d\mathfrak{s}
+\left(\frac{4p}{p+2}-\frac{2}{3}\right)\int_{\gamma}\Psi^2\sigma^{-\frac{6p}{p+2}}\sigma_{\mathfrak{s}}^2d\mathfrak{s}\\
&-\frac{1}{3}\left(1-\frac{3p}{p+2}\right)^2
\int_{\gamma}\Psi^2\sigma^{-\frac{6p}{p+2}}\sigma_{\mathfrak{s}}^2d\mathfrak{s}
-\frac{2}{3}\left(1-\frac{3p}{p+2}\right)
\int_{\gamma}\Psi\Psi_{\mathfrak{s}}\sigma^{1-\frac{6p}{p+2}}\sigma_{\mathfrak{s}}d\mathfrak{s}\\
&-\frac{1}{3}\int_{\gamma}\Psi_{\mathfrak{s}}^2\sigma^{2-\frac{6p}{p+2}}d\mathfrak{s}.
\end{align*}

Adding up $I$, $II$ and $III$, we obtain equation (\ref{e: 6}).
\end{proof}
\begin{lemma}\label{lem: w p-ratio}
The weighted $p$-affine isoperimetric ratio, $\frac{\Omega_p^{\Psi}}{A^{\frac{2-p}{2+p}}}$, is non-decreasing along the flow (\ref{e: 2}) and remains constant if and only if
$K_t$ is a homothetic solution to the flow.
\end{lemma}
\begin{proof}
Using equation $\sigma_{\mathfrak{s}\mathfrak{s}}+\sigma\mu=1$ which relates the affine curvature $\mu$ to the affine support function, we rewrite the first term in (\ref{e: 6}) as follows:
\begin{equation}\label{e: 5}
\frac{2(p-2)}{p+2}\int_{\gamma}\Psi^2\sigma^{1-\frac{6p}{p+2}}d\mathfrak{s}=
\frac{2(p-2)}{p+2}\int_{\gamma}\Psi^2\sigma^{2-\frac{6p}{p+2}}\mu
d\mathfrak{s}+
\frac{2(p-2)}{p+2}\int_{\gamma}\Psi^2\sigma^{1-\frac{6p}{p+2}}\sigma_{\mathfrak{s}\mathfrak{s}}
d\mathfrak{s}.
\end{equation}
On the other hand, by the affine- geometric Wirtinger inequality Lemma 6, \cite{BA1}, we have
 \begin{equation}
 \int_{\gamma}\Psi^2\sigma^{2-\frac{6p}{p+2}}\mu d\mathfrak{s}\leq \frac{1}{2A}\left(\int_{\gamma}\Psi\sigma^{1-\frac{3p}{p+2}} d\mathfrak{s}\right)^2+\int_{\gamma}\left(\Psi\sigma^{1-\frac{3p}{p+2}}\right)_{\mathfrak{s}}^2d\mathfrak{s}.
 \end{equation}
 Therefore, by equation (\ref{e: 5}), we get
 \begin{equation}\label{ie: 4}
 \int_{\gamma}\Psi^2\sigma^{1-\frac{6p}{p+2}}d\mathfrak{s}\leq \frac{1}{2A}\left(\int_{\gamma}\Psi\sigma^{1-\frac{3p}{p+2}} d\mathfrak{s}\right)^2+\int_{\gamma}\left(\Psi\sigma^{1-\frac{3p}{p+2}}\right)_{\mathfrak{s}}^2d\mathfrak{s}
+\int_{\gamma}\Psi^2\sigma^{1-\frac{6p}{p+2}}\sigma_{\mathfrak{s}\mathfrak{s}}
d\mathfrak{s}.
 \end{equation}

We also have
\begin{align}\label{e: expan1}
\int_{\gamma}\left(\Psi\sigma^{1-\frac{3p}{p+2}}\right)_{\mathfrak{s}}^2d\mathfrak{s}&=\left(1-\frac{3p}{p+2}\right)^2
\int_{\gamma}\Psi^2\sigma^{-\frac{6p}{p+2}}\sigma_{\mathfrak{s}}^2d\mathfrak{s}+
\int_{\gamma}\Psi_{\mathfrak{s}}^2\sigma^{2-\frac{6p}{p+2}}d\mathfrak{s}\nonumber\\
&+
2\left(1-\frac{3p}{p+2}\right)
\int_{\gamma}\Psi\Psi_{\mathfrak{s}}\sigma^{1-\frac{6p}{p+2}}\sigma_{\mathfrak{s}}d\mathfrak{s},
\end{align}
and
\begin{align}\label{e: expan2}
\int_{\gamma}\Psi^2\sigma^{1-\frac{6p}{p+2}}\sigma_{\mathfrak{s}\mathfrak{s}}d\mathfrak{s}
=-2\int_{\gamma}\Psi_{\mathfrak{s}}\Psi\sigma^{1-\frac{6p}{p+2}}\sigma_{\mathfrak{s}}d\mathfrak{s}
+\left(\frac{6p}{p+2}-1\right)\int_{\gamma}\Psi^2\sigma^{-\frac{6p}{p+2}}\sigma_{\mathfrak{s}}^2d\mathfrak{s}.
\end{align}
Hence by combining equation (\ref{e: 6}), inequality (\ref{ie: 4}), equations (\ref{e: expan1}), (\ref{e: expan2}) and collecting similar terms, we obtain
 \begin{align*}
 \frac{d}{d t}\Omega_{p}^{\Psi}&\geq
    \left(\frac{p-2}{p+2}\right)\frac{1}{A}\left(\int_{\gamma}\Psi\sigma^{1-\frac{3p}{p+2}} d\mathfrak{s}\right)^2
   +\frac{18p^2(p-1)}{(p+2)^3}\int_{\gamma}\Psi^2\sigma^{-\frac{6p}{p+2}}\sigma_{\mathfrak{s}}^2d\mathfrak{s}\\\nonumber
   &+\frac{2(p-1)}{p+2}\int_{\gamma}\Psi_{\mathfrak{s}}^2\sigma^{2-\frac{6p}{p+2}}d\mathfrak{s}
   -\frac{12p(p-1)}{(p+2)^2}\int_{\gamma}\Psi\Psi_{\mathfrak{s}}\sigma^{1-\frac{6p}{p+2}}\sigma_{\mathfrak{s}}d\mathfrak{s}.\nonumber
 \end{align*}
 Now, we observe that the last three terms in the previous inequality can be grouped in a term that is almost a perfect square:
 \begin{align*}
 \frac{d}{d t}\Omega_{p}^{\Psi}&\geq
   \left(\frac{p-2}{p+2}\right)\frac{1}{A}\left(\int_{\gamma}\Psi\sigma^{1-\frac{3p}{p+2}} d\mathfrak{s}\right)^2
   +\frac{9p^2}{2(p^2+p-2)}\int_{\gamma}\left(\Psi^{\frac{2(p-1)}{3p}}\sigma^{1-\frac{3p}{p+2}}\right)_{\mathfrak{s}}^2
   \Psi^{\frac{2(p+2)}{3p}}d\mathfrak{s}.
 \end{align*}
To finish the proof, we note that by (3) in Lemma \ref{lem: 1} and the previous inequality, we
have
\begin{align}\label{e: p-ratio}
\frac{d}{d t}\frac{\Omega_p^{\Psi}}{A^{\frac{2-p}{2+p}}}(t)
=\frac{1}{A^{\frac{2-p}{2+p}}(t)}\left(\frac{d}{dt}\Omega_p^{\Psi}-\frac{p-2}{p+2}\frac{(\Omega_p^{\Psi})^2}{A}\right)(t)\ge
0\nonumber.
\end{align}
 \end{proof}
\begin{lemma}\label{cor: limsup idea} If $K_t$ evolves by (\ref{e: 2}), the following limit holds as $t$ approaches the extinction time $T$:
\begin{equation}
 \lim\inf_{t\to T} \frac{(\Omega_p^{\Psi})^p}{A^{1-p}}\left[\frac{d}{dt}\Omega_p^{\Psi}-\frac{p-2}{p+2}\frac{(\Omega_p^{\Psi})^2}{A}\right]=0.
 \label{eq:sup}
 \end{equation}
\end{lemma}
\begin{proof}
We have
\begin{align*}
\frac{d}{d t}\frac{(\Omega_p^{\Psi})^{2+p}}{A^{2-p}}(t)
&=-\frac{d}{d
t}\ln(A(t))\left[\frac{(\Omega_p^{\Psi})^p}{A^{1-p}}\left((2+p)\frac{d}{dt}\Omega_p^{\Psi}-(p-2)\frac{(\Omega_p^{\Psi})^2}{A}\right)(t)\right].
\end{align*}
If
$$\frac{(\Omega_p^{\Psi})^p}{A^{1-p}}\left[(2+p)\frac{d}{dt}\Omega_p^{\Psi}-(p-2)\frac{(\Omega_p^{\Psi})^2}{A}\right]\geq\varepsilon$$
in a neighborhood of $T$, then
$$\frac{d}{d t}\frac{(\Omega_p^{\Psi})^{2+p}}{A^{2-p}}(t)\ge -\varepsilon \frac{d}{d t}\ln(A(t)).$$
Thus,
$$\frac{(\Omega_p^{\Psi})^{2+p}}{A^{2-p}}(t)\ge \frac{d}{d t}\frac{(\Omega_p^{\Psi})^{2+p}}{A^{2-p}}(t_1)+\varepsilon \ln(A(t_1))-\varepsilon \ln(A(t)),$$
the right hand side goes to infinity as $A(t)$ goes to zero. This contradicts the $p$-affine isoperimetric inequality which states that the left hand side is bounded from above.
\end{proof}

\section{The Normalized flow}
In this section, we study the asymptotic behavior of the evolving curves under a normalized flow corresponding to the evolution described by (\ref{e: 1}). We consider the conventional rescaling such that the area enclosed by the normalized curves is $\pi$ by taking
 $$\tilde{s}_t:=\sqrt{\frac{\pi}{A(t)}}\,s_t,~~~\tilde{\kappa}_t:=\sqrt{\frac{A(t)}{\pi}}\,{\kappa}_t.$$
 One can also define a new time parameter
 $$\tau=\int_{0}^t\left(\frac{\pi}{A(K_t)(\xi)}\right)^{\frac{2p}{p+2}}d\xi$$
 and can easily verify that
\begin{equation}\label{e: normalized flow}
\frac{\partial }{\partial \tau}\,
\tilde{s}=-\Psi\tilde{s}\left(\frac{\tilde{\kappa}}{\tilde{s}^3}\right)^{\frac{p}{p+2}}+
\frac{\Psi\tilde{s}}{2\pi}\,\tilde{\Omega}_p^{\Psi},
\end{equation}
where $\tilde{\Omega}_p^{\Psi}$ stands for the weighted $p$-affine length of $\partial \tilde{K}_t$ having support function $\tilde{s}_t$. More precisely,
$$\tilde{\Omega}_p^{\Psi}(\tau):=\Omega_p^{\Psi}(\tilde{K}_{\tau})=
\int_{\mathbb{S}^1}\Psi\frac{\tilde{s}}{\tilde{\kappa}}\left(\frac{\tilde{\kappa}}{\tilde{s}^3}\right)^{\frac{p}{p+2}}d\theta.$$
However, even in the normalized case, we prefer to work on the finite time interval $[0,T)$.
\begin{corollary} \label{cor: limit of affine support} Let $\{t_k\}_k$ be the sequence of times realizing the limit (\ref{eq:sup}) in Lemma \ref{cor: limsup idea}. Then, there exists a constant $c>0$ such that along the normalized $p$-flow, we have
$$\lim_{t_k\to T}\Psi^{\frac{2(p-1)}{3p}}\sigma^{1-\frac{3p}{p+2}}(t_k)=c.$$
\end{corollary}
\begin{proof}
By Lemma \ref{cor: limsup idea}, we have
$$0=\lim_{t_k\to T} \frac{(\Omega_p^{\Psi})^p}{A^{1-p}}\left[\frac{d}{dt}\Omega_p^{\Psi}-\frac{p-2}{p+2}\frac{(\Omega_p^{\Psi})^2}{A}\right]\geq \lim_{t_k\to T}\frac{c_p(\Omega_p^{\Psi})^{p}}{A^{1-p}}\int_{\gamma}\left(\Psi^{\frac{2(p-1)}{3p}}\sigma^{1-\frac{3p}{p+2}}\right)_{\mathfrak{s}}^2
   \Psi^{\frac{2(p+2)}{3p}}d\mathfrak{s},$$
where $c_p:=\frac{9p^2}{2(p+2)(p-1)}.$
As by Lemma \ref{lem: w p-ratio}, the normalized weighted $p$-affine length $\tilde{\Omega}^{\Psi}_p$ is increasing along the normalized flow and $\Psi$ has a lower bound, we conclude that
$$\lim_{t_k\to T}\int_{\tilde{\gamma}}\left(\Psi^{\frac{2(p-1)}{3p}}\tilde{\sigma}^{1-\frac{3p}{p+2}}\right)_{\tilde{\mathfrak{s}}}^2d\tilde{\mathfrak{s}}=0.$$
We note that, for any $\theta_1, \theta_2 \in {\mathbb{S}^1}$,
\begin{align*}
\left|\int_{\theta_1}^{\theta_2}\left(\Psi^{\frac{2(p-1)}{3p}}\tilde{\sigma}^{1-\frac{3p}{p+2}}\right)_{\theta}\,d\theta
\right| &\leq
\int_{\mathbb{S}^1}\left|\left(\Psi^{\frac{2(p-1)}{3p}}\tilde{\sigma}^{1-\frac{3p}{p+2}}\right)_{\theta}\right|d\theta \\ &=\int_{\tilde{\gamma}}\left|\left(\Psi^{\frac{2(p-1)}{3p}}\tilde{\sigma}^{1-\frac{3p}{p+2}}\right)_{\tilde{\mathfrak{s}}}\right|d\tilde{\mathfrak{s}}\\
&\leq \left(\int_{\tilde{\gamma}}\left(\Psi^{\frac{2(p-1)}{3p}}
\tilde{\sigma}^{1-\frac{3p}{p+2}}\right)_{\tilde{\mathfrak{s}}}^2d\tilde{\mathfrak{s}}\right)^{1/2}
\tilde{\Omega}_1^{1/2}.
\end{align*}
Take $\theta_1$ and $\theta_2$ be two points where $\Psi^{\frac{2(p-1)}{3p}}\tilde{\sigma}^{1-\frac{3p}{p+2}}$ reaches its extremal values. It is known that, for a smooth, simple curve with enclosed area $\pi$, $\min_{\mathbb{S}^1}\sigma \leq 1$ and $\max_{\mathbb{S}^1}\sigma \geq 1$, see Lemma 10 in \cite{BA1}. Hence, as $\tilde{\Omega}_1$ is bounded from above by the classical affine isoperimetric inequality \cite{Lutwak}, we infer that
$$\lim_{t_k\to T}\Psi^{\frac{2(p-1)}{3p}}\tilde{\sigma}^{1-\frac{3p}{p+2}}(t_k)=c,$$ for some constant $c.$
\end{proof}
 The following lemma will be needed in the proof of the main theorem.
\begin{lemma}\label{lem: john n}
Let $s$ be the support function of a $\frac{\pi}{k}$ $(k\geq 2)$ periodic, smooth convex curve $\gamma$ of enclosed area $\pi$. Then there exist uniform lower and upper bounds on $s$ depending only on $k$.
\end{lemma}
\begin{proof}
We write the cosine series of $s$, $s(\theta)=s_0+\sum_{n=1}^{\infty} s_n\cos(2nk\theta).$ From this, we conclude that we can represent radius of curvature $\mathfrak{r}$ as follows:
$$\mathfrak{r}=s_0+\sum_{n=1}^{\infty}(1-4n^2k^2)s_{n}\cos(2nk\theta)>0.$$
We will use now the positivity of $\mathfrak{r}$ to find an estimate for the upper bound of $|s_n|.$
$$\int_{\mathbb{S}^1}\mathfrak{r}(1\pm \cos(2nk\theta))d\theta=2\pi s_0\pm \pi s_n(1-4n^2k^2),$$
thus we have
\begin{equation}\label{e: upper coef s}
|s_n|\leq \frac{2s_0}{4n^2k^2-1}~, ~~\forall n\geq 1.
\end{equation}
To find an upper bound for $s$, we use the assumption that $\gamma$ encloses an area of $\pi$ and inequality (\ref{e: upper coef s}).
\begin{align*}
2\pi=\int_{\mathbb{S}^1}\mathfrak{r} sd\theta&=2\pi s_0^2-\pi\sum_{n=1}^{\infty}(4n^2k^2-1)s_n^2\\
&\geq 2\pi\left(1-2\sum_{n=1}^{\infty}\frac{1}{4n^2k^2-1}\right)s_0^2\\
&=2\pi\left(\frac{\pi\cot(\frac{\pi}{2k})}{2k}\right)s_0^2
=:2\pi \frac{1}{c_k}s_0^2.
\end{align*}
Hence, we have
$$s_0\leq \sqrt{c_k}.$$
 On the other hand, we have
\begin{equation}\label{e: upper bound}
s(\theta)=\frac{1}{2}\left(s(\theta)+s(\theta+\pi)\right)\leq \frac{1}{4}L(\gamma)=\frac{1}{2}\pi s_0\leq \frac{\pi\sqrt{c_k}}{2}.
\end{equation}
To find a lower bound for $s$, we use the assumption that $\gamma$ encloses an area of $\pi$, inequality (\ref{e: upper bound}) and the maximal ellipsoid contained in the convex body enclosed by $\gamma.$ Let $J$ denotes the maximal ellipsoid (also known as the John ellipsoid) contained in the convex body enclosed by $\gamma.$ It is know that
\begin{equation}\label{ie: john1}
 J\subset\gamma\subset \sqrt{2}J,
 \end{equation}
 see \cite{John}.
Therefore, $J$ encloses an area of, at least, $\frac{\pi}{2}.$ Suppose the major axis of  $J$ has length $l_1$ and the minor axis of $J$ has length $l_2$.
Hence,
\begin{equation}\label{ie: john2}
\frac{\pi}{2}\leq \pi\frac{l_1l_2}{4}=A(J).
\end{equation}
On the other hand, by (\ref{ie: john1}) and inequality (\ref{e: upper bound}), we know that $l_1\leq\frac{\pi\sqrt{c_k}}{4}.$
Now, as $l_1l_2>2$ by (\ref{ie: john2}), we conclude that $l_2> \frac{2}{\pi\sqrt{c_k}}.$ Once again using (\ref{ie: john1}) implies
$$s(\theta)\geq \frac{l_2}{2}>\frac{1}{\pi\sqrt{c_k}}.$$
\end{proof}

\section{Proof of the main theorem}
In this section we present a proof of the main theorem.
\begin{proof}
Define $\Phi=:\Psi^{\frac{p+2}{3p}}$ in (\ref{e: 1}). Then an appropriate rescaling of the evolving convex bodies and Corollary \ref{cor: limit of affine support} prove the first part of the claim. To prove the second part, we start the flow (\ref{e: 1}) with a curve that whose support function is $\frac{\pi}{k}$ periodic; for example $s(\theta, 0):=1+\varepsilon\cos(2k\theta)$ for $\varepsilon>0$ small enough. Therefore, the solution to the evolution equation (\ref{e: 1}) remains $\frac{\pi}{k}$-periodic. Hence, by the Lemma \ref{lem: john n}, $\tilde{s}$ is bounded. Therefore, Corollary \ref{cor: limit of affine support} and the standard theory of parabolic equations imply the claim.
\end{proof}

We remark that the periodicity of $\Phi$ with period $\frac{\pi}{k}$, $k\geq2$, was also considered in a different way by Chen \cite{Chen}
as a sufficient condition for the solvability of the $L_{-2}$ Minkowski problem.

\section{A necessary condition and the uniqueness of solutions}
In this section, we obtain a necessary condition on the solvability of the even $ L_{-2}$ Minkowski problem, hence showing that the existence of solutions to the problem cannot occur for all $\pi$-periodic smooth functions $\Psi$. Moreover, we will use the initial set up of this section to discuss the uniqueness of solutions to the even $ L_{-2}$ Minkowski problem.
\begin{theorem}\label{thm: necessary condition}
Let $\gamma$ be a smooth, origin-symmetric curve. Assume ${\gamma}:\mathbb{S}^1\to\mathbb{R}^2,$ is the Gauss parametrization of $\gamma$. Then, $\sigma $, the affine support function of $\gamma$, as a function on the unit circle has at least eight critical points, i.e., points at which $\sigma_{\theta}=0.$
\end{theorem}

\begin{proof}
Define a curve $\Lambda:\mathbb{S}^1\to\mathbb{R}$ by
$$\Lambda(\theta):=\left(\int_{0}^{\theta}\frac{\cos\alpha}{s^3(\alpha)}~d\alpha,\int_{0}^{\theta}\frac{\sin\alpha}{s^3(\alpha)}~d\alpha\right)
=:(x,y).$$
As $\gamma$ is origin symmetric,  $s(\theta+\pi)=s(\theta)$ for all $\theta\in\mathbb{S}^1$. This implies $\Lambda(2\pi)=\Lambda(0)=\vec{o}$ and that $\Lambda$ is a closed curve. For convenience set $~~':=\frac{d}{d\theta}$. We compute the Euclidean curvature of $\Lambda$:
 $$\kappa_\Lambda= \frac{|x'y''-y'x''|}{(x'^2+y'^2)^{3/2}}=s^3.$$
Hence, $\Lambda$ is a closed convex curve. We now proceed to obtain the affine curvature of $\Lambda$ using the following formula
$$\mu_{\Lambda}=\underbrace{\frac{x''y'''-x'''y''}{(x'y''-x''y')^{5/3}}}_{i}\underbrace{-\frac{1}{2}\left[\frac{1}{(x'y''-x''y')^{2/3}}\right]''}_{ii}.$$

We have
\begin{equation*}
i=\frac{\left(\frac{\cos\theta}{s^3}\right)'\left(\frac{\sin\theta}{s^3}\right)''-\left(\frac{\cos\theta}{s^3}\right)''\left(\frac{\sin\theta}{s^3}\right)'}
{\left(\frac{\sin\theta}{s^3}\left(\frac{\cos\theta}{s^3}\right)'-\left(\frac{\sin\theta}{s^3}\right)'\frac{\cos\theta}{s^3}\right)^{5/3}}
=s^2\left(3ss''+6s'^2+s^2\right)
\end{equation*}
and
\begin{equation*}
ii=-\frac{1}{2}\left(\frac{1}{\left(\frac{\sin\theta}{s^3}\left(\frac{\cos\theta}{s^3}\right)'-
\left(\frac{\sin\theta}{s^3}\right)'\frac{\cos\theta}{s^3}\right)^{2/3}}\right)''
=-2s^2\left(ss''+3s'^2\right).
\end{equation*}
Adding up $i$ and $ii$  gives $\displaystyle\mu_{\Lambda}=s^3\left(s_{\theta\theta}+s\right)=\frac{s^3}{\kappa}=\sigma^3.$
It is known, see for example \cite{Bushin}, that a centrally symmetric oval has at least 8 extatic points, i.e., points where $\mu_{\theta}=0$.
Therefore, $\sigma$ must have, at least, eight critical points.
\end{proof}
\begin{corollary}
If the even $L_{-2}$ Minkowski problem with smooth data $\Psi$ has a solution, then $\Psi$ must have $8$, or more, critical points on $[0, 2\pi]$.
\end{corollary}
\begin{proposition}
Let $\gamma_1$ and $\gamma_2$ be two smooth, origin-symmetric curves with support functions $s_1$ and $s_2$, respectively.
If $\sigma_{\gamma_1}\equiv\sigma_{\gamma_2}=:\Phi$, then there exists a special linear transformation, $T\in \textbf{SL}(2)$ such that
$$\gamma_2=T(\gamma_1).$$
Furthermore, identifying $\theta$ and $(\cos\theta,\sin \theta)$  we have
\begin{equation}\label{e: phi after trans}
\Phi(\theta)=\Phi\left(\frac{T^{-t}(\cos\theta,\sin \theta)}{||T^{-t}(\cos\theta,\sin \theta)||}\right).
\end{equation}
\end{proposition}
\begin{proof}
It is well-known that affine curvature determines a curve uniquely up to an equiaffine transformation of the plane, \cite{NS}. Define
$$\Lambda_1(\theta):=\left(\int_{0}^{\theta}\frac{\cos\alpha}{s_1^3(\alpha)}~d\alpha,\int_{0}^{\theta}\frac{\sin\alpha}{s_1^3(\alpha)}~d\alpha\right),
\Lambda_2(\theta):=\left(\int_{0}^{\theta}\frac{\cos\alpha}{s_2^3(\alpha)}~d\alpha,\int_{0}^{\theta}\frac{\sin\alpha}{s_2^3(\alpha)}~d\alpha\right).
$$

Since $\mu_{\Lambda_1}\equiv \mu_{\Lambda_1}$, there exists a special linear transformation $T\in \textbf{SL}(2)$ such that
$$\Lambda_2=T(\Lambda_1).$$
Let $\vec{\textbf{n}}_{\Lambda_1}$ and $\vec{\textbf{n}}_{\Lambda_2}$ denote the unit normal to, respectively, $\Lambda_1$ and $\Lambda_2$.
Therefore, for any $x\in \Lambda_1$
$$\kappa_{\Lambda_1}(x)=||T^{-t}(\vec{\textbf{n}}_{\Lambda_1}(x))||^3\kappa_{\Lambda_2}(T(x)).$$
On the other hand, using $\kappa_{\Lambda_1}(x)=s_1^3(x)$ and $\kappa_{\Lambda_2}(T(x))=s_2^3(T(x))$, we obtain that
\begin{equation}\label{e: support after trans}
s_1^3(x)=||T^{-t}(\vec{\textbf{n}}_{\Lambda_1}(x))||^3s_2^3(T(x)).
\end{equation}
To prove the corollary we need to rewrite the  equation (\ref{e: support after trans}) on the unit circle. Toward that end, we observe the following relation between
$\vec{\textbf{n}}_{\Lambda_1}$ and $\vec{\textbf{n}}_{\Lambda_2}$:
 $$\vec{\textbf{n}}_{\Lambda_2}=\frac{T^{-t}(\vec{\textbf{n}}_{\Lambda_1})}{||T^{-t}(\vec{\textbf{n}}_{\Lambda_1})||}.$$
Thus, $s_1(\vec{\textbf{n}}_1)=||T^{-t}(\vec{\textbf{n}}_1)||s_2(\vec{\textbf{n}}_2),$ where $\vec{\textbf{n}}_1, \vec{\textbf{n}}_2\in\mathbb{S}^1$
and $\vec{\textbf{n}}_2=\frac{T^{-t}(\vec{\textbf{n}}_1)}{||T^{-t}(\vec{\textbf{n}}_1)||}.$ This completes the proof of the first part. The proof of equation (\ref{e: phi after trans}) also follows from the above observations.
\end{proof}
\noindent \textbf{Remark:} \medskip \\
Suppose that the even $L_{-2}$ Minkowski problem is solvable for $\Phi$. Equivalently, there exists a curve $\gamma$ such that $\frac{s^3}{\kappa}=\Phi$. If $\Phi(\theta)=\Phi\left(\frac{T^{-t}(\cos\theta,\sin \theta)}{||T^{-t}(\cos\theta,\sin \theta)||}\right)$ then it is easy to show that $T(\gamma)$ also solves the even $L_{-2}$ Minkowski problem corresponding to $\Phi.$
This fact and the previous Corollary imply that any curve in
$$\left\{T(\gamma);~ T\in \textbf{SL}(2,\mathbb{R})~ \mbox{and}~ \Phi(\theta)=\Phi\left(\frac{T^{-t}(\cos\theta,\sin \theta)}{||T^{-t}(\cos\theta,\sin \theta)||}\right)\right\}$$
solves $\frac{s^3}{\kappa}=\Phi$ and these are all the possible solutions.

\section{Conclusions}

We will recall first some results of Ai, Chou and Wei,  \cite{Chou1}, who employed a different sufficiency condition in their study of the $L_{-2}$ problem.\\
Let $\Phi:\mathbb{S}^1\to \mathbb{R}$ be a smooth positive function. Define
$$B(x,\Phi):=\int_{0}^{\pi}\frac{\Phi(x+t)-\Phi(x)-2^{-1}\Phi'(x)\sin(2t)}{\sin^2t}dt.$$
If $B(x)\neq 0$ at any critical point of $\Phi$ then, we say $\Phi$ is $B$-nondegenerate.\\
\textbf{Theorem A} \cite{Chou1} Assume that $\Phi$ is a positive, $B$-nondegenerate, $C^2$ function of period $\pi.$ Then there exist a constant $C$ which depends only on $\Phi$ such that
$$0<C^{-1}\leq u(x)\leq C,\ \ \ {\hbox{and}}\ \ \  ||u||_{H^1(\mathbb{S}^1)}\leq C,$$ for any solution $u$ of the $L_{-2}$ Minkowski corresponding to $\Phi.$\\
\textbf{Theorem B} \cite{Chou1} Assume that $\Phi$ is a positive, $B$-nondegenerate, $C^2$ function of period $\pi.$ Then $L_{-2}$ Minkowski problem with data $\Phi$
is solvable  if the winding number of the map
$$x\mapsto(-B(x), \Phi'(x)),~~x\in[0,\pi)$$
around the origin is not equal to $-1$.\\
\textbf{Lemma 1.5.} (Kazdan-Warner type obstruction) \cite{Chou1} For any solution $u$ of the $L_{-2}$ Minkowski problem corresponding to $\Phi$, we have
$$\int_{0}^{\pi}\frac{\Phi'(x)\alpha(x)}{u^2(x)}dx=0$$
where $\alpha$ is in the set $\{1,\cos 2x,\sin 2x\}$.

Define $C:=\{\Phi\in C^{\infty}_{even}(\mathbb{S}^1); \Phi>0~ \mbox{and}~ \exists u: u''+u=\frac{\Phi}{u^3}\}.$ Then by our the main theorem, $S$ is dense in $D:=\{\Phi\in C^{\infty}_{even}(\mathbb{S}^1), \Phi>0\}$ with respect to the $L^{\infty}$ norm. By Corollary \ref{thm: necessary condition}, or the Kazdan-Warner type obstruction, if $\Phi$ is only $\pi$ periodic, then it is possible that the corresponding $L_{-2}$ is not solvable. A simple example is provided by $\Phi(\theta)=2+\cos(2\theta).$ For any non-solvable $\Phi$, Theorem A and Theorem \ref{thm: main1} imply that there exists a family of convex bodies such that their corresponding affine support functions are $B$-degenerate while approaching $\Phi$ in the $L^{\infty}$ norm. Therefore, the $B$-non degeneracy of $\Phi$ is \emph{not} a necessary condition for the existence of a solution to the $L_{-2}$ Minkowski problem. Moreover, note that by a result of Guggenheimer, \cite{guggenheimer}, the above lemma also implies that $\Psi_{\theta}$ has, at least, $8$ zeroes in $[0,2\pi]$, hence 8 critical points is not a sufficient condition as Kazdan-Warner type obstruction is not a sufficient condition.\\\\

\textbf{Acknowledgment.}
The author is immensely indebted to Alina Stancu, for her encouragements, comments and suggestions throughout this work.

\bibliographystyle{amsplain}

\end{document}